\documentclass{article}
\usepackage{amsmath,amsfonts,amssymb,mathrsfs,amsthm,tikz-cd,url,hyperref}  
\usepackage{todonotes}
\newtheorem{thm}{Theorem}[section]

\theoremstyle{plain}
\newtheorem{prop}[thm]{Proposition}

\newtheorem{lem}[thm]{Lemma}
\newtheorem{cor}[thm]{Corollary}

\theoremstyle{definition}
\newtheorem{defn}[thm]{Definition}

\theoremstyle{remark}
\newtheorem{rem}[thm]{Remark}

\title{Mono-anabelian Reconstruction of Number Fields with Restricted Ramification}
\author{Yu Mao\\
\texttt{maoyu@westlake.edu.cn}
\and
Xiao Wang \\ 
\texttt{wangxiao24@westlake.edu.cn}
}
\date{}
\begin{document}
\maketitle
\begin{abstract}
In this paper, we apply Hoshi's results in \cite{Ho1} and \cite{Ho2}, to establish a group-theoretic reconstruction of a number field $K$ together with its maximal unramified outside $S$ extension for some density 1 set of prime $S$ starting from the profinite group $G_{K,S}$.
\end{abstract}
\newpage
\tableofcontents
\newpage
\section*{Conventions}
Throughout this paper, we shall use the following conventions. \\
\textbf{Sets}
\begin{itemize}
    \item We shall denote by $\mathbb N$ for the set of positive integers. Moreover, we shall denote by $\mathfrak{Primes}$ for the set of prime numbers.
    \item Let $\Sigma \subset \mathfrak{Primes}$ be a non-empty set of prime numbers, we shall write $\mathbb N(\Sigma)$ for the set of $\Sigma$-positive integers, which are positive integers whose prime divisors belong to $\Sigma$, and $1$ is always regarded as a $\Sigma$-integer for any non-empty $\Sigma$.
\end{itemize}
\textbf{Groups}
\begin{itemize}
    \item Let $G$ be a profinite group, we shall write $G^{\text{ab}}$ for its maximal abelian quotient, $G^{\text{sol}}$ for its maximal pro-solvable quotient.
    \item Let $G$ be an abelian group, we shall write $G_{\text{tor}}$ for the torsion subgroup of $G$. If moreover, $G$ is profinite, then $G_{\text{tor}}$ denotes the closure of the torsion subgroup in $G$. Moreover, for each integer $n$, we shall write $G[n]$ for the $n$-torsion subgroup of $G_{\text{tor}}$. 
    \item All homomorphisms between profinite groups are assumed to be continuous. 
    \item Let $G$ be a group, we shall denote by $G^{\Sigma}$ for the pro-$\Sigma$ completion of $G$ for some non-empty subset $\Sigma \subset \mathfrak{Primes}$. Moreover, if $G$ is abelian, we shall write $G^{\wedge} := \varprojlim_n G/nG$ (when written additively).
\end{itemize}
\textbf{Rings and Fields}
\begin{itemize}
    \item All rings throughout this paper are assumed to be commutative and unital.
    \item Let $R$ be a ring, we denote by $R^{\times}$ for the unit group of $R$. If $K$ is a field, we shall write $K_{\times}$ for the multiplicative monoid of $K$, i.e. $K$ with its additive structure forgotten.
    \item Let $K$ be a field. We shall write $\mu(K)$ for the roots of unity contained $K$, moreover, we shall write $\Lambda(K) := \varprojlim_n ~\mu(K)[n]$ where $n \in \mathbb N$.
    \item Let $K$ be a field. We shall write $\overline{K}$ for a fixed algebraic closure of $K$, $K^{\text{sol}}$ for the maximal pro-solvable extension of $K$ contained in $\overline K$, moreover, we write $K^{\text{ab}}$ for the maximal abelian extension of $K$ contained in $\overline{K}$.
    \item Let $K$ be a number field, we shall write $\mathscr{P}_K$ for the set of all primes of $K$, moreover, we write $\mathscr{P}_K^{\text{inf}}$ for the set of archimedean primes of $K$ and $\mathscr{P}_{K}^{\text{fin}}$ for the set of non-archimedean primes of $K$. 
    \item Let $K$ be a number field, we shall write $\mathscr{P}_K^{\bullet} := \{\mathfrak{p} \in \mathscr{P}_K^{\text{fin}}: \text{char}(\mathfrak{p}) = 2~\text{or}~e_{\mathfrak{p}}>1\}$ where $\text{char}(\mathfrak{p})$ is defined to be the residue characteristic of $\mathfrak{p}$ and $e_{\mathfrak{p}}$ denotes the ramification index of $\mathfrak{p}$. 
    \item Let $K$ be a number field, and let $\mathfrak{p} \in \mathscr{P}_K$ be a prime, we shall write $K_{\mathfrak{p}}$ for the completion of $K$ at $\mathfrak{p}$.
    \item Let $K$ be a number field, and let $S \subset \mathscr{P}_K$, we shall write $K_S$ for the maximal unramified outside $S$ extension of $K$, which is the union of all finite extension $L/K$ which is unramified outside $S$. We shall write $\mathcal{O}_{K,S}$ for the ring of $S$-integers, i.e. the subring of $K$ consisting of elements with non-negative valuations at primes outside $S$. Furthermore, we shall write $\mathbb N(S) := \{n \in \mathbb N: n \in \mathcal O_{K,S}^{\times}\}$. Moreover, we shall denote by $\delta(S)$ for the natural density of $S$. 
    \item Let $K$ be a number field, and let $S \subset \mathscr{P}_K$, we shall write $S_f$ for the subset of all non-archimedean primes in $S$.
    \item Let $K$ be a number field, and let $S \subset \mathscr{P}_K$ be a non-empty subset containing at least one finite prime. We say that a set of prime number $\Sigma$ is determined by $S$ or $S$ determines $\Sigma$ if $\Sigma = \{\ell \in \mathfrak{Primes}: \exists\mathfrak{p} \in S~\text{s.t.}~\text{char}(\mathfrak{p}) = \ell\}$.
    \item Let $K$ be a field, $K^{\text{sep}}$ the separable closure contained in $\overline{K}$, we write $G_K:=\text{Gal}(K^{\text{sep}}/K)$ for the absolute Galois group of $K$. If $K$ is a number field, and $S \subset \mathscr{P}_K$ is a subset of primes of $K$, we write $G_{K,S} := \text{Gal}(K_S/K)$.
\end{itemize}
\section{Introduction}
Anabelian geometry is roughly speaking, the study of schemes via their \'etale fundamental groups. The first result in anabelian geometry was proven in 1970s, known as the Neukirch-Uchida theorem:
\begin{thm}[The Neukirch-Uchida theorem]
Let $K,L$ be number fields. Then the natural map
$$
\text{Isom}(L^{\text{sol}}/L,K^{\text{sol}}/K) \to \text{Isom}(G_K^{\text{sol}},G_L^{\text{sol}})/\text{Inn}(G_L^{\text{sol}})
$$
is bijective.
\end{thm}
Later, Hoshi in \cite{Ho1} and \cite{Ho2} established a group-theoretic algorithm, to recover a number field $K$ together with its maximal pro-solvable extension $K^{\text{sol}}$ group-theoretically starting from $G_K^{\text{sol}}$, i.e.
\begin{thm}[Hoshi, \cite{Ho1} and \cite{Ho2}]
Let $G$ be a profinite group isomorphic to $G_{K}^{\text{sol}}$ for some number field $K$. Then there is a group-theoretic reconstruction to a solvably closed field $\widetilde{F}(G)$ starting from $G$, and the fixed subfield $F(G) := \widetilde{F}(G)^G$ such that: \par 
(i) There exists an isomorphism $G \xrightarrow{\sim} \text{Gal}(\widetilde{F}(G)/F(G))$.\par 
(ii) The following diagram commutes
$$
\begin{tikzcd}
\widetilde{F}(G) \arrow[r,"\sim"] & K^{\text{sol}} \\
F(G) \arrow[u,hookrightarrow] \arrow[r,"\sim"] & K \arrow[u,hookrightarrow]
\end{tikzcd}
$$
where the top horizontal arrow is a Galois-equivariant isomorphism, vertical arrows are field embeddings and the bottom horizontal arrow is an isomorphism.
\end{thm}
In anabelian geometry, we often call results like Theorem 1.1 bi-anabelian results, which roughly speaking, are results starting from an isomorphism between Galois groups/\'etale fundamental groups, ending up with an isomorphism between fields/schemes. \par 
On the other hand, we call results like Theorem 1.2 mono-anabelian results, which roughly speaking, are results starting from an abstract profinite group $G$ which is isomorphic to the (or some quotient of) \'etale fundamental group of some scheme $X$, ending up with a purely group-theoretic reconstruction $G \mapsto X(G)$ of a scheme $X(G)$ isomorphic to $X$. \par 
A natural way to generalise Theorem 1.1 is to replace $G_K^{\text{sol}}$ by $G_{K,S}$ for suitable set of primes $S \subset \mathscr{P}_K$, which was proven by Shimizu in \cite{Shi1} and \cite{Shi2}:
\begin{thm}[Shimizu, \cite{Shi2} Theorem 2.4]
Let $K,L$ be number fields, and let $S_K,S_L$ be set of primes of $K,L$ respectively. Assume that the following conditions hold: \par 
\begin{itemize}
    \item[(1)] The set of prime numbers determined by $S_K$ (resp. $S_L$) has cardinality at least $2$.
    \item[(2)] For any finite Galois extension $K_0/K$ or $L_0/L$ contained in $K_{S_K}$ or $L_{S_L}$, the natural density of the intersection between the set of prime numbers determined by $S_K$ or $S_L$ and the set of prime numbers splits completely in $K_0$ or $L_0$ is non-zero.
    \item[(3)] Assume condition (2). There exists a prime number $\ell$ in the intersection of the set of prime numbers determined by $S_K$ and the set of prime numbers determined by $S_L$, $S_L$ or $S_K$ satisfying $\star_{\ell}$ condition respectively (c.f. Definition 1.16 in \cite{Shi1}).
\end{itemize}
Then the natural map
$$
\text{Isom}(K_{S_K}/K,L_{S_L}/L) \to \text{Isom}(G_{L,S_{L}},G_{K,S_K})
$$
is a bijection.
\end{thm}
The main result of this paper is to develop a mono-anabelian version of Theorem 1.3 in the density 1 case (since conditions (1),(2),(3) in Theorem 1.3 holds true automatically in density $1$):
\begin{thm}[c.f. Theorem 5.6]
Let $G$ be an abstract profinite group, isomorphic to $G_{K,S}$ for some number field $K$ and $S \subset \mathscr{P}_K$ satisfying the following conditions: \par 
\begin{itemize}
    \item $\delta(S) = 1$.
    \item $\mathscr{P}_K^{\text{inf}}$ and $\mathscr{P}_K^{\bullet}$ are contained in $S$. 
    \item $S$ is conjugate-stable, i.e. if $\mathfrak{p} \in \mathscr{P}_K$ lies in $S$, so does all conjugates of $\mathfrak{p}$. 
\end{itemize}
Then there exists a group-theoretic reconstruction of a field $F_S(G)$ together with the action of $G$ on $F_S(G)$ by automorphism, and the fixed subfield $F(G) := F_S(G)^G$ with Galois group $\text{Gal}(F_S(G)/F(G)) \xrightarrow{\sim} G$, such that the following diagram commutes:
$$
\begin{tikzcd}
F_S(G) \arrow[r,"\sim"] & K_S \\
F(G) \arrow[r,"\sim"] \arrow[u,hookrightarrow] & K \arrow[u,hookrightarrow]
\end{tikzcd}
$$
where the top horizontal arrow is an isomorphism equivariant w.r.t $G \xrightarrow{\sim} G_{K,S}$, the vertical arrows are field embeddings and the bottom arrow is an isomorphism.
\end{thm} 
\textbf{Structure of this paper}: \par 
In section 2, we use the group-theoretic reconstruction of decomposition groups at primes in $S$ as proved in Proposition 2.1 in \cite{Shi1}, we apply Theorem 1.4 in \cite{Ho1} to recover the local invariants at primes in $S$. \par 
In section 3, we use the local reconstructions established in section 2 to recover $\mu(K_S)$ as a $G_{K,S}$-module, hence also the $G_{K,S}$-module $\Lambda(K_S)$. Furthermore, we construct a suitable container $\mathcal{H}^{\times}(K,S)$ to contain $\mathcal{O}_{K,S}^{\times}$ as a subgroup. Notice that we are unable to recover the $S$-unit group directly from $G_{K,S}$. \par 
In section 4, we recover the maximal pro-solvable and unramified outside $\Sigma$ extension of $\mathbb Q$ group-theoretically starting from $G_{K,S}$, hence also the profinite group $G_{\mathbb Q,\Sigma}^{\text{sol}}$. Also, we recover the set of decomposition groups of $G_{\mathbb Q,\Sigma}^{\text{sol}}$ at primes lying above $\Sigma$. \par 
In section 5, we finish the reconstruction algorithm by characterising suitable subfields of the algebraic closure of completions of $K$ at primes in $S$. In this step, the full-ness of decomposition groups at primes lying above $S$ is crucial. Finally, by applying Theorem 2.4 in \cite{Shi2}, we obtain natural isomorphisms between those subfields, and recover a global copy which is isomorphic to $K_S$.  \par 
We follow the strategy in \cite{Ho1} and \cite{Ho2}, but technically the proof of Theorem 1.4 in quite different from \cite{Ho2}. \par 
\textbf{Further aspects}:\par 
In \cite{CC} Remarque 5.3 (i), it was shown that for totally real fields $K$, if the set of prime numbers determined by $S$ has cardinality at least $2$, then decomposition groups of $G_{K,S}$ at primes above $S$ are full, hence one of the obstruction in Remark 5.8 is not an issue. In this case, we may consider proving a positive density version of Theorem 1.4 for totally real fields.
\section{Reconstruction of Local Invariants}
In this section, we develop a group-theoretic reconstruction of local invariants as in Section 1 in \cite{Ho1}. 
\begin{defn}
Let $G$ be a profinite group. We say that $G$ is of NF-type with restricted ramification if there exists a collection of data as follows: \par 
(i) A number field $K$. \par
(ii) A subset $S \subset \mathscr{P}_K$ such that $\mathscr{P}_K^{\text{inf}} \subset S$ and $\mathscr{P}_K^{\bullet} \subset S$ where $\mathscr{P}_K^{\bullet} := \{\mathfrak{p} \in \mathscr{P}_K^{\text{fin}}:\text{char}(\mathfrak{p} )=2~\text{or}~e_{\mathfrak{p}} > 1\}$. \par
(iii) The set $S$ is conjugate-stable, that is for any prime number $\ell \in \mathfrak{Primes}$, if $\mathfrak{p} \in S$ is lying above $\ell$, then all primes of $k$ lying above $\ell$ are also belong to $S$. \par 
(iv) The maximal unramified outside $S$ extension $K_S$ of $K$. \par
(v) An isomorphism $\alpha:G \xrightarrow{\sim}G_{K,S} = \text{Gal}(K_S/K)$. \par 
If $\delta(S) = 1$, we we say that $G$ is of NF-type with density $1$ restricted ramification.
\end{defn}
One checks immediately that open subgroups of profinite groups of NF-type with density 1 restricted ramification are again profinite groups of NF-type with density 1 restricted ramification.
\begin{prop}
Let $G$ be a profinite group of NF-type with density 1 restricted ramification. Then there is a group-theoretic reconstruction of the $G$-set
$$
\widetilde{S}(G)
$$
such that there is an $\alpha$-equivariant bijection
$$
\widetilde{S}(G) \xrightarrow{\sim} \{D_{\bar{\mathfrak{p}}} \subset G_{K,S}: \bar{\mathfrak{p}} \in S_f(K_S)\}.
$$
Moreover, we shall write $S(G) := \widetilde{S}(G)/G$. In this case, the following diagram commutes
$$
\begin{tikzcd}
\widetilde{S}(G) \arrow[r,"\sim"] \arrow[d,twoheadrightarrow] & S_f(K_S) \arrow[d,twoheadrightarrow] \\
S(G) \arrow[r,"\sim"] & S.
\end{tikzcd}
$$
\end{prop}
\begin{proof}
This follows from Proposition 1.20 and Proposition 2.1 in \cite{Shi1}. Or alternatively, one combines Corollary 2.7 (ii) in \cite{Ivanov} and Theorem 9.4.3 in \cite{NSW}, and arguments in the proof of Theorem 12.1.9 in \cite{NSW}, one can conclude that the set of decomposition subgroups of $G_{K,S}$ at primes in $S_f(K_S)$ are maximal closed subgroups of MLF-type (c.f. Definition 1.1 in \cite{Ho1}).

\end{proof}
\begin{rem}
By Theorem 9.4.3 in \cite{NSW}, it holds that elements in $\{D_{\bar{\mathfrak{p}}} \subset G_{K,S}: \bar{\mathfrak{p}} \in S^{\text{fin}}(K_S)\}$ are full, i.e. are absolute Galois groups of $p$-adic local fields for suitable $p$.
\end{rem}
\begin{thm}
Let $G$ be a profinite group of NF-type with density 1 restricted ramification. Let $G_{K_{\mathfrak{p}}}\xleftarrow{\sim}D \in \widetilde{S}(G)$. There is a group-theoretic reconstruction of the following objects from $D$: \par 
(i) The residue characteristic of $K_{\mathfrak{p}}$, denoted by $p(D)$.\par 
(ii) The degree $[K_{\mathfrak{p}}:\mathbb{Q}_{p(D)}]$, denoted by $d(D)$. \par 
(iii) The inertia degree of $K_{\mathfrak{p}}$, denoted by $f(D)$. \par 
(iv) The absolute ramification index of $K_{\mathfrak{p}}$, denoted by $e(D)$. \par 
(v) The inertia subgroup of $G_{K_{\mathfrak{p}}}$, denoted by $I(D)$. \par 
(vi) The wild inertia subgroup of $G_{K_{\mathfrak{p}}}$, dentoed by $W(D)$. \par 
(vii) The Frobenius element of $G_{K_{\mathfrak{p}}}^{\text{unr}}$, denoted by $\text{Frob}(D)$. \par 
(viii) The unit group $\mathcal{O}_{K_{\mathfrak{p}}}^{\times}$ of $K_{\mathfrak{p}}$, denoted by $\mathcal{O}^{\times}(D)$. \par 
(ix) The multiplicative group $K_{\mathfrak{p}}^{\times}$, denoted by $k^{\times}(D)$. \par 
(x) The group of roots of unity contained in $K_{\mathfrak{p}}^{\text{sep}}$, denoted by $\mu(D)$. \par 
(xi) The local cyclotome $\Lambda(K_{\mathfrak{p}}^{\text{sep}})$, denoted by $\Lambda(D)$.
\end{thm}
\begin{proof}
This follows immediately from Theorem 1.4 in \cite{Ho1} (together with Remark 2.3).
\end{proof}
\section{Synchronisation of $\Sigma$-Cyclotomes}
In this section, we develop a version of local-global sychronisation like Theorem 3.8 in our context. In particular, we are not able to recover the full cyclotome, but we can recover the $\Sigma$-part of the global cyclotome for the set of prime numbers $\Sigma$ determined by $S$. \par 
\begin{defn}
Let $G$ be a profinite group of NF-type with density 1 restricted ramification. Let $v \in S(G)$ (c.f. Proposition 2.2). For each $D \in \widetilde{S}(G)$ lying above $v$, we define
\begin{itemize}
    \item $p_v:= p(D)$, $d_v:= d(D)$, $f_v := f(D)$ and $e_v := e(D)$.
    \item $k^{\times}(v) := k^{\times}(D)$.
    \item $\mathcal{O}^{\times}(v) := \mathcal{O}^{\times}(D)$.
\end{itemize}
In particular, $\mathcal{O}^{\times}(v)$ and $k^{\times}(v)$ are independent from the choice of $D$ lying above $v$ (c.f. Proposition 3.7 in \cite{Ho1}). 
\end{defn}
\begin{defn}
Let $G$ be a profinite group of NF-type with density 1 restricted ramification. We write
$$
\Sigma(G) := \{p \in \mathfrak{Primes}: \exists v \in S~s.t.~p_v=p\}.
$$
\end{defn}
\begin{defn}
Let $K$ be a number field, and let $S \subset \mathscr{P}_K$ be a subset with $\delta(S) = 1$ and $\mathscr{P}_K^{\text{inf}} \subset S$. We write
$$
\mathcal{I}(K,S) := \bigcup_{S'}(\prod_{v \in S'} k^{\times}(v) \times \prod_{v \in S_f\setminus S'} \mathcal{O}^{\times}(v))
$$
where $S'$ ranges over all finite subsets of $S_f$.
\end{defn}
\begin{lem}
Let $K$ be a number field, and let $\mathscr{P}_K^{\text{inf}} \subset S \subset \mathscr{P}_K$ be a subset with $\delta(S) = 1$. Then the following assertions hold true: \par
(i) There exists a natural map
$$
\mathcal{I}(K,S) \to G_{K,S}^{\text{ab}}
$$
obtained by forming the composite
$$
\mathcal{I}(K,S) \hookrightarrow \mathbb I_K^{\text{fin}} \hookrightarrow \mathbb I_K \twoheadrightarrow C_K \twoheadrightarrow G_K^{\text{ab}} \twoheadrightarrow G_{K,S}^{\text{ab}}
$$
which coincide with the natural map $\mathcal{I}(K,S) \to G_{K,S}^{\text{ab}}$ obtained by forming the product of local Artin reciprocity maps $\theta_{\mathfrak{p}}: K_{\mathfrak{p}}^{\times} \to G_{K_{\mathfrak{p}}}^{\text{ab}}$ for each $\mathfrak{p} \in S_f$. \par
(ii) It holds that
$$
\text{ker}(\mathcal{I}(K,S)\to G_{K,S}^{\text{ab}})_{\text{tor}} \subset \mu(K).
$$
Moreover, if $K$ is totally imaginary, then
$$
\text{ker}(\mathcal{I}(K,S) \to G_{K,S}^{\text{ab}})_{\text{tor}} = \mu(K).
$$
(iii) It holds that
$$
\mu(K_S) = \varinjlim_L ~\text{ker}(\mathcal{I}(L,S) \to G_{K,S}^{\text{ab}})_{\text{tor}}
$$
where $L$ ranges over all finite extensions of $K$ contained in $K_S$, and the transition maps are induced by transfer maps (c.f. Proposition 1.5.9 in \cite{NSW}). In particular, there is a canonical identification
$$
\varinjlim_L~\text{ker}(\mathcal{I}(K,S) \to G_{K,S}^{\text{ab}})_{\text{tor}} \xrightarrow{\sim} \bigoplus_{\ell \in \Sigma} \mathbb{Q}_{\ell}/\mathbb{Z}_{\ell} (1).
$$
\end{lem}
\begin{proof}
Assertion (i) follows immediately from global class field theory. Now we verify assertion (ii). \par 
Let $a \in \text{ker}(\mathcal{I}(K,S) \to G_{K,S}^{\text{ab}})_{\text{tor}}$. If the image of $a$ in $G_K^{\text{ab}}$ is trivial, then it follows immediately from Lemma 3.6 (iii) in \cite{Ho1} that $a \in K^{\times}$. Now suppose that the image of $a$ in $G_K^{\text{ab}}$ is non-trivial. By Theorem 8.3.13 of \cite{NSW}, there exists an idele $\gamma\in\mathbb{I}_{K}$ which is $1$ at the primes in $S$ and units at the primes not in $S$, such that $a\gamma$ has image in $D_{K}$, where $D_{K}$ is the kernel of the global reciprocity map $C_{K}\twoheadrightarrow G_{K}^{\text{ab}}$. By \cite{NSW} , Theorem 8.2.5, we can write 
$$a\gamma=x\cdot\beta\prod_{i}\epsilon_{i}^{\lambda_{i}},$$
where $\beta$ is a purely complex idele, $x\in K^{\times}$, $\epsilon_{i}\in\mathcal{O}_{K}^{\times}$ are viewed as ideles in $\mathbb{I}_{K}$ via the natural embedding $K^{\times}\to\mathbb{I}_{K}$, $\lambda_{i}\in\widehat{\mathbb{Z}}$. Then for each finite prime $\mathfrak{p}$ we have $x\in U_{\mathfrak{p}}$, therefore $x\in\mathcal{O}_{K}^{\times}$. The group generated by the $\epsilon_{i}$'s has finite index in $\mathcal{O}_{K}^{\times}$, hence there exists a positive integer $d$ such that $x^{d}=\prod_{i}\epsilon_{i}^{n_{i}}$ with $n_{i}\in\mathbb{Z}$. We can choose $d$ such that $a^{d}=1$. Hence for $\mathfrak{p}\in S$ we have
$$1=\alpha^{d}=x^{d}\prod_{i}\epsilon_{i}^{d\lambda_{i}}=\prod_{i}\epsilon_{i}^{d\lambda_{i}+n_{i}}.$$
The $\epsilon_{i}$'s, viewed as elements of $\prod_{\mathfrak{p}\in S^{\text{fin}}}U_{\mathfrak{p}}$, are $\widehat{\mathbb{Z}}$-independent, this follows from the proof of \cite{NSW}, Lemma 8.2.4 and our assumption that $\delta(S)=1$. Therefore we see $\lambda_{i}\in\mathbb{Z}$, and hence $a \in K^{\times}$. \par 
Now assume that $K$ is totally imaginary. Let $b \in K^{\times}$ be a torsion. We need to prove that the image of $b$ in $G_{K,S}^{\text{ab}}$ is trivial. $b$, viewed as as an element of $\mathbb{I}_{K}$ via the natural embedding $K^{\times}\to\mathbb{I}_{K}$, has image in $D_{K}$, and its difference with $b$, viewed as an element of $\mathcal{I}(K,S)$, is an idele with value $b$ at the components of infinite primes and finite primes not in $S$. Since $K$ is totally imaginary, the infinite primes are all complex. By Theorem 8.2.5 and Theorem 8.3.13 of \cite{NSW}, we see the image of $b$ in $G_{K,S}^{\text{ab}}$ is trivial. \par 
Assertion (iii) follows immediately from assertion (ii) together with the fact that for any finite extension $L/K$ contained in $K_S$, there exists some totally imaginary finite extension $L'/K$ contained in $K_S$ such that $L'$ contains $L$.
\end{proof}
\begin{defn}
Let $G$ be a profinite group of NF-type with density 1 restricted ramificaiton. We define
$$
\mathcal{I}(G,S(G)) := \bigcup_{S'} (\prod_{v \in S'} k^{\times}(v) \times \prod_{v \in S(G) \setminus S'} \mathcal{O}^{\times}(v))
$$
where $S' \subset S(G)$ ranges over all finite subsets of $S(G)$.
\end{defn}
\begin{rem}
It follows immediately from Theorem 2.4 that, for a profinite group $G$ of NF-type with density 1 restricted ramification, the isomorphism $\alpha$ induces an isomorphism
$$
\mathcal{I}(G,S(G)) \xrightarrow{\sim} \mathcal{I}(K,S).
$$
\end{rem}
\begin{defn}
Let $G$ be a profinite group of NF-type with density 1 restricted ramification. We define
$$
\mu_{\Sigma(G)}(G) := \varinjlim_H ~\text{ker}(\mathcal{I}(H,S(H)) \to H^{\text{ab}})_{\text{tor}}
$$
where the transition maps are determined by the transfer maps. Moreover, we define the $\Sigma(G)$-cyclotome associated to $G$ as
$$
\Lambda_{\Sigma(G)}(G) := \varprojlim_n ~\mu_{\Sigma(G)}(G)[n]
$$
where the inverse limit is taken over all positive integers $n$.
\end{defn}
\begin{thm}[Synchronisation of $\Sigma(G)$-cyclotomes]
Let $G$ be a profinite group of NF-type with density 1 restricted ramification. Let $D \in \widetilde{S}(G)$. Then the following assertions hold: \par 
(i) Let $H \subset G$ be an open subgroup, then the natural inclusion $H \hookrightarrow G$ induces the following $H$-equivariant isomorphisms
$$
\mu_{\Sigma(G)}(G) \xrightarrow{\sim} \mu_{\Sigma(H)}(H)~\text{and}~\Lambda_{\Sigma(G)}(G) \xrightarrow{\sim} \Lambda_{\Sigma(H)}(H).
$$
(ii) The natural surjection $\mathcal{I}(G,S(G)) \twoheadrightarrow k^{\times}(D)$ induces the following $D$-equivariant isomorphisms
$$
\mu_{\Sigma(G)}(G) \xrightarrow{\sim} \bigoplus_{\ell \in \Sigma(G)}\mu_{\ell^{\infty}}(D)~\text{and}~ \Lambda_{\Sigma(G)}(G) \xrightarrow{\sim} \prod_{\ell \in \Sigma(G)} \Lambda_\ell(D)
$$
where $\mu_{\ell^{\infty}}(D) := \varinjlim_n (\varinjlim_H k^{\times}(H))[\ell^n]$ for $H$ ranges over all open subgroups of $D$ and $\Lambda_\ell(D) := \varprojlim_n \mu_{\ell^{\infty}}(D)[n]$.
\end{thm}
\begin{proof}
Assertion (i) follows immediately from the fact that an open subgroup of NF-type with density 1 restricted ramification is again a profinite group of NF-type with density 1 restricted ramification, together with the fact that $\Sigma(G) = \Sigma(H)$ and Definition 3.7. \par 
Now we verify assertion (ii). Notice that to verify assertion (ii), it suffices to verify that for each $\ell \in \Sigma(G)$, there exists a $D$-equivariant isomorphism
$$
\mu_{\ell^{\infty}}(G) \xrightarrow{\sim} \mu_{\ell^{\infty}}(D)
$$
where
$$
\mu_{\ell^{\infty}}(G) := \varinjlim_n ~\mu_{\Sigma(G)}(G)[\ell^n]
$$
and
$$
\mu_{\ell^{\infty}}(D) := \varinjlim_n ~\mu(D)[\ell^n].
$$
Notice that there is a natural map
$$
\varinjlim_n ~\varinjlim_H ~\text{ker}(\mathcal{I}(H,S(H)) \to H^{\text{ab}})[\ell^n] \to \varinjlim_n ~\varinjlim_H~k^{\times}(D_H)[\ell^{n}]
$$
where $H$ ranges over all open subgroups of $G$ and $D_H := D \cap H$. In particular, we have a natural map
$$
t: \mu_{\ell^{\infty}}(G) \xrightarrow{\sim} \mu_{\ell^{\infty}}(D).
$$
The injectivity of $t$ follows from Lemma 3.4 (iii). The surjectivity of $t$ follows from the injectivity of $t$, since both $\mu_{\ell^{\infty}}(G)$ and $\mu_{\ell^{\infty}}(D)$ are abstractly isomorphic to $\mathbb{Q}_{\ell}/\mathbb{Z}_{\ell}$ and every injective endomorphism of $\mathbb{Q}_{\ell}/\mathbb{Z}_{\ell}$ is surjective. Moreover, one checks immeidately that $t$ is $D$-equivariant. This proves Theorem 3.8.
\end{proof}
\begin{cor}
Let $G$ be a profinite group of NF-type with density 1 restricted ramification. Let $D \in \widetilde{S}(G)$. Then the natural surjection $\mathcal{I}(G,S(G)) \twoheadrightarrow k^{\times}(D)$ induces the following $D$-equivariant inclusion
$$
\iota_D: \Lambda_{\Sigma(G)}(G) \hookrightarrow \Lambda(D).
$$
\end{cor}
\begin{proof}
This is immediate from Theorem 3.8 (ii).
\end{proof}
Now let $K$ be a number field, and let $S \subset \mathscr{P}_K$ be an infinite subset containing $\mathscr{P}_K^{\text{inf}}$ such that $\delta(S) = 1$. We have the following exact sequence (c.f. Proposition 8.3.4 in \cite{NSW})
$$
1 \to \mu_n \to \mathcal{O}_S^{\times} \xrightarrow{(-)^n} \mathcal{O}_S^{\times} \to 1
$$
where $\mathcal{O}_S$ denotes the union $\bigcup_{L/K} \mathcal{O}^{\times}_{L,S}$ for all finite extension $L/K$ contained in $K_S$, and $n \in \mathbb{N}(S)$. In particular, we have an exact sequence
$$
1 \to \mathcal{O}_{K,S}^{\times}/(\mathcal{O}_{K,S}^{\times})^n \to H^1(G_{K,S},\mu_n) \to \text{Cl}_S(K)[n] \to 1.
$$
But since $\delta(S) = 1$, it follows from Chebotarev density theorem that $\text{Cl}_S(K) = 1$, hence we obtain an isomorphism
$$
\mathcal{O}_{K,S}^{\times}/(\mathcal{O}_{K,S}^{\times})^n \xrightarrow{\sim} H^1(G_{K,S},\mu_n).
$$
\begin{defn}
Let $K$ be a number field, and let $S \subset \mathscr{P}_K$ be a subset such that $\delta(S) = 1$ and $\mathscr{P}_K^{\text{inf}} \subset S$. Moreover, we write $\Sigma$ for the set of prime numbers determined by $S$. We define the $S$-Kummer container associated to $K$ to be the fibre product in the following diagram
$$
\begin{tikzcd}
& \mathcal{I}(K,S) \arrow[d,hookrightarrow] \\
(\mathcal{O}_{K,S}^{\times})^{\Sigma} \arrow[r,hookrightarrow] & \prod_{\mathfrak{p} \in S} (K_{\mathfrak{p}}^{\times})^{\wedge}
\end{tikzcd}
$$
where the horizontal arrow is defined to be the composite
$$
(\mathcal{O}_{K,S}^{\times})^{\Sigma} \hookrightarrow (K^{\times})^{\Sigma} \hookrightarrow \prod_{\mathfrak{p} \in S_f}(K_{\mathfrak{p}}^{\times})^{\Sigma} \hookrightarrow \prod_{\mathfrak{p} \in S_f}(K_{\mathfrak{p}}^{\times})^{\wedge}.
$$
In this case, we denote by $\mathcal{H}^{\times}(K,S)$ for the $S$-Kummer container associated to $K$. Moreover, we write $\mathcal{H}_{\times}(K,S) := \mathcal{H}^{\times}(K,S) \cup \{0\}$.
\end{defn}
\begin{lem}
Let $K$ be a number field and let $S \subset \mathscr{P}_K$ be a subset such that $\delta(S) = 1$ and $\mathscr{P}_K^{\text{inf}} ~\text{and}~\mathscr{P}_K^{\bullet}\subset S$ and that $S$ is conjugate-stable. Moreover, we write $\Sigma$ for the set of prime numbers determined by $S$. Then the followings hold \par 
(i) The following diagram has exact rows and commutes
$$
\begin{tikzcd}
	1 & {\mathcal{O}_K^{\times}} & {\mathcal{O}_{K,S}^{\times}} & {\mathcal{O}_{K,S}^{\times}/\mathcal{O}_{K}^{\times}} & 1 \\
	1 & {(\mathcal{O}_K^{\times})^{\Sigma}} & {\mathcal{H}^{\times}(K,S)} & {\mathcal{O}_{K,S}^{\times}/\mathcal{O}_K^{\times}} & 1.
	\arrow[from=1-1, to=1-2]
	\arrow[from=1-2, to=1-3]
	\arrow[hook, from=1-2, to=2-2]
	\arrow[from=1-3, to=1-4]
	\arrow[hook, from=1-3, to=2-3]
	\arrow[from=1-4, to=1-5]
	\arrow[no head, from=1-4, to=2-4]
	\arrow[shift left, no head, from=1-4, to=2-4]
	\arrow[from=2-1, to=2-2]
	\arrow[from=2-2, to=2-3]
	\arrow[from=2-3, to=2-4]
	\arrow[from=2-4, to=2-5]
\end{tikzcd}
$$
(ii) The middle injective map in the diagram displayed in assertion (i) induces an inclusion:
$$
\mu(K) \hookrightarrow \mathcal{H}^{\times}(K,S)_{\text{tor}}
$$
and an isomorphism
$$\mu(K)^{\Sigma} \xrightarrow{\sim} \mathcal{H}^{\times}(K,S)_{\text{tor}}^{\Sigma}$$
(iii) We write $S_f^{d=1}$ for the subset of $S_f$ consisting of elements of degree $1$ (c.f. $d_v = 1$ in Definition 3.1). Then the composite
$$
\mathcal{H}_{\times}(K,S) \to \prod_{\mathfrak{p} \in S_f} (K_{\mathfrak{p}})_{\times} \twoheadrightarrow \prod_{\mathfrak{p} \in S_f^{d=1}} (K_{\mathfrak{p}})_{\times}
$$
is injective.
\end{lem}
\begin{proof}
For assertion (i), the exactness of the top sequence is trivial. We shall verify the exactness of the bottom sequence. \par 
Consider the natural exact sequence
$$
1 \to \mathcal{O}_K^{\times} \to \mathcal{O}_{K,S}^{\times} \to J_{K,S}
$$
where $J_{K,S} := \text{Div}(\text{Spec}(\mathcal{O}_{K,\mathscr{P}_K \setminus S}))$. \par
Then it follows from the proof of Lemma 5.29 (i) in \cite{Ho4} (by replacing various $k^{\times}$ by $\mathcal{O}_{k,S}^{\times}$), that pro-$\Sigma$ completion induces the following exact sequence.
$$
1 \to (\mathcal{O}_{K}^{\times})^{\Sigma} \to (\mathcal{O}_{K,S}^{\times})^{\Sigma} \to J_{K,S}^{\Sigma}.
$$
In particular, to verify the exactness of the lower sequence, it suffice to verify that the inverse image of $\mathcal{O}_{K,S}^{\times}/\mathcal{O}_{K}^{\times} \subset J_{K,S} \subset J_{K,S}^{\Sigma}$ in $(\mathcal{O}^{\times}_{K,S})^{\Sigma}$ coincides with $\mathcal{H}^{\times}(K,S)$. \par 
Consider the natural map
$$
\mathcal{I}(K,S) \to J_{K,S}
$$
$$
(\dots,a_{\mathfrak{p}},\dots) \mapsto \prod_{\mathfrak{p} \in S_f} \mathfrak{p}^{\text{ord}_{\mathfrak{p}}(a_{\mathfrak{p}})}
$$
determines natural map
$$
\prod_{\mathfrak{ p} \in S} (K_{\mathfrak{p}}^{\times})^{\wedge} \to J_{K,S}^{\Sigma}
$$
$$
(\dots,u_{\mathfrak{p}}\pi_{\mathfrak{p}}^{z_{\mathfrak{p}}},\dots) \mapsto \prod_{\mathfrak{p} \in S} \mathfrak{p}^{z_{\mathfrak{p}}}
$$
where $z_{\mathfrak{p}} \in \mathbb Z^{\Sigma}$. In particular, we have the following commutative diagram
$$
\begin{tikzcd}
	& {(\mathcal{O}^{\times}_{K,S})^{\Sigma}} \\
	{\prod_{\mathfrak{p}\in S}(K_{\mathfrak{p}}^{\times})^{\wedge}} & {J_{K,S}^{\Sigma}} \\
	{\mathcal{I}({K,S})} & {J_{K,S}}.
	\arrow[hook, from=1-2, to=2-1]
	\arrow[from=1-2, to=2-2]
	\arrow[from=2-1, to=2-2]
	\arrow[hook', from=3-1, to=2-1]
	\arrow[from=3-1, to=3-2]
	\arrow[hook', from=3-2, to=2-2]
\end{tikzcd}
$$
Let us write $\widetilde{P}$ for the inverse image of $\mathcal{O}_{K,S}^{\times}/\mathcal{O}_K^{\times} \subset J_{K,S}$ in $\mathcal{I}(K,S)$ via the lower horizontal map in the above diagram. We have
$$
\widetilde{P} = \{(\dots,a_{\mathfrak{p}},\dots): \exists x \in \mathcal{O}_{K,S}^{\times}~\text{s.t.}~\text{ord}_{\mathfrak{p}}(a_{\mathfrak{p}}) = \text{ord}_{\mathfrak{p}}(x)~\forall \mathfrak{p} \in S_f\}.
$$
Then one verifies immediately that the intersection of $(\mathcal{O}_{K,S}^{\times})^{\Sigma}$ and $\widetilde{P}$ in $\prod_{\mathfrak{p} \in S_f}(K_{\mathfrak{p}}^{\times})^{\wedge}$ coincide with $\mathcal{H}^{\times}(K,S)$. Hence we can conclude that
$$
1 \to (\mathcal{O}_K^{\times})^{\Sigma} \to \mathcal{H}^{\times}(K,S) \to \mathcal{O}^{\times}_{K,S}/\mathcal{O}_K^{\times} \to 1
$$
is exact. The injectivity of vertical arrows are immediate. This proves assertion (i) and assertion (ii) is an immediate consequence of assertion (i) together with the condition $\mathscr{P}_K^{\bullet} \subset S$. \par
Now we verify assertion (iii). Notice that we have $\delta(S) = \delta(S_f) = \delta(S_f^{d=1})$ (c.f. Chapter VII discussion above Theorem 13.2 in \cite{ANT}, together with the fact that finite intersections of density 1 sets is again of density 1). Then it follows from Theorem 9.1.11 in \cite{NSW}, that the following composite
$$
(\mathcal{O}_{K,S}^{\times})^{\Sigma} \hookrightarrow (K^{\times})^{\Sigma} \to \prod_{\mathfrak{p} \in S_f^{d=1}}(K_{\mathfrak{p}}^{\times})^{\Sigma} \to \prod_{\mathfrak{p} \in S_f^{d=1}}(K_{\mathfrak{p}}^{\times})^{\wedge}
$$
is injective. Hence also, $\mathcal{H}_{\times}(K,S) \to \prod_{\mathfrak{p} \in S_f} (K_{\mathfrak{p}})_{\times} \twoheadrightarrow \prod_{\mathfrak{p} \in S_f^{d=1}}(K_{\mathfrak{p}})_{\times}$ is injective.
\end{proof}
\begin{lem}
Let $k$ be a $p$-adic local field for some prime number $p$. Let $\Sigma$ be a non-empty set of prime numbers. Then there exists a natural inclusion
$$
H^1(G_k,\mathbb Z^{\Sigma}(1)) \to H^1(G_k,\widehat{\mathbb{Z}}(1))
$$
induced by the following $G_k$-equivariant inclusion:
$$
\mathbb Z^{\Sigma}(1) \hookrightarrow \widehat{\mathbb Z}(1).
$$
\end{lem}
\begin{proof}
Consider the following exact sequence of $G_k$-modules:
$$
0 \to \mathbb Z^{\Sigma}(1) \to \widehat{\mathbb{Z}}(1) \to \widehat{\mathbb Z}(1)^{(\Sigma')} \to 0.
$$
Then this lemma follows immediately by taking the long exact sequence to the above sequence together with the fact that $H^0(G_k,\mathbb Z_{\ell}(1)) = 1$ for any prime number $\ell$.
\end{proof}

\begin{defn}
Let $G$ be a profinite group NF-type with density 1 restricted ramification. We define the $S(G)$-Kummer container associated to $G$ as the fibre product in the following diagram
$$
\begin{tikzcd}
    & \mathcal{I}(G,S(G)) \arrow[d,hookrightarrow] \\
    H^1(G,\Lambda_{\Sigma(G)}(G)) \arrow[r,hookrightarrow] & \prod_{D \in \widetilde{S}(G)} H^1(D,\Lambda(D))
\end{tikzcd}
$$
where the vertical arrow is defined to be the composite (c.f. Proposition 3.11 in \cite{Ho1})
$$
\mathcal{I}(G,S(G)) \hookrightarrow \prod_{v \in S(G)} k^{\times}(v) \hookrightarrow \prod_{D \in \widetilde{S}(G)} k^{\times}(D) \xrightarrow{\prod_{D \in \widetilde{S}(G)}\text{Kmm}_D} \prod_{D \in \widetilde{S}(G)} H^1(D,\Lambda(D))
$$
and the horizontal arrow is defined to be the composite (the well-defined-ness and natural-ness follow from Theorem 3.8 and Corollary 3.9 and Lemma 3.12)
$$
H^1(G,\Lambda_{\Sigma(G)}(G)) \hookrightarrow \prod_{D\in \widetilde{S}(G)} H^1(D,\prod_{\ell \in \Sigma(G)}\Lambda_{\ell}(D)) \hookrightarrow \prod_{D \in \widetilde{S}(G)}H^1(D,\Lambda(D)).
$$
We shall denote by $\mathcal{H}^{\times}(G,S(G))$ for this group. Moreover, we define $\mathcal{H}_{\times}(G,S(G)) := \mathcal{H}^{\times}(G,S(G)) \cup \{0\}$.
\end{defn}

\begin{prop}
Let $G$ be a profinite group of NF-type with density 1 restricted ramification. Let $H\subset G$ be an open subgroup. Then there exists a commutative diagram whose vertical arrows are injective, horizontal arrows are isomorphisms:
$$
\begin{tikzcd}
	{\mathcal{H}^{\times}(G,S(G))} & {\mathcal{H}^{\times}(K,S)} \\
	{\mathcal{H}^{\times}(H,S(H))} & {\mathcal{H}^{\times}(L,S_L)}
	\arrow["\sim", from=1-1, to=1-2]
	\arrow[from=1-1, to=2-1]
	\arrow[from=1-2, to=2-2]
	\arrow["\sim", from=2-1, to=2-2]
\end{tikzcd}
$$
where $L/K$ is the finite extension corresponding to $H$, and $S_L$ is the set of primes of $L$ lying above primes in $S$. Moreover, if $H$ is open normal in $G$, then the lower isomorphism is equivariant w.r.t $G/H \xrightarrow{\sim} \text{Gal}(L/K)$ induced by the isomorphism $\alpha$ (c.f. Definition 2.1).
\end{prop}
\begin{proof}
To verify that the horizontal arrows are isomorphisms, it suffice to verify that we have the following isomorphisms
$$
\mathcal{I}(G,S(G)) \xrightarrow{\sim} \mathcal{I}(K,S)~;~H^1(G,\Lambda_{\Sigma(G)}(G)) \xrightarrow{\sim} (\mathcal{O}_{K,S}^{\times})^{\Sigma}.
$$
On the other hand, the first isomorphism above follows from Remark 3.6 and the second isomorphism above follows from Kummer theory. The injectivity of the vertical arrows are immediate from various definitions involved. The commutativity and the second assertion are immediate.
\end{proof}
\section{Reconstruction of $\mathbb{Q}_{\Sigma}^{\text{sol}}$}
In this section, we shall establish a group-theoretic algorithm to reconstruct a field isomorphic to the maximal unramified outside $\Sigma$ pro-solvable extension of $\mathbb Q$ (i.e. $\mathbb Q_{\Sigma}^{\text{sol}}$) starting from a profinite group $G$. 
\begin{rem}
Notice that $\mathbb Q_{\Sigma}$ is not prosolvable if $\Sigma \neq \emptyset$. Indeed, we just take $\Sigma := \{p\}$ to be the set consisting of a single prime number. Then $G_{\mathbb Q,\{p\}}$ is not pro-solvable for any prime number $p$, e.g. see \cite{Ram2}.
\end{rem}

\begin{lem}
Let $K$ be a number field, and let $S$ be a set of primes of $K$ such that $\mathscr{P}_K^{\text{inf}},\mathscr{P}_K^{\bullet} \subset S$ and $\delta(S) = 1$, furthermore $S$ is conjugate-stable. Moreover, we write $\Sigma$ for the set of prime numbers determined by $S$. Then for each $a \in \mathcal{H}^{\times}(K,S)$ such that there exists $n \in\mathbb N(S)$ with $a^n \in \mathcal{O}_{K,S}^{\times}$, it holds that $a \in \mathcal{O}_{K,S}^{\times}$.
\end{lem}
\begin{proof}
Let $a \in \mathcal{H}^{\times}(K,S)$ be such that there exists $n \in \mathbb N(S)$ with $a^n \in \mathcal{O}_{K,S}^{\times}$. It follows from Lemma 3.11 (i), that we have the equality
$$
\text{coker}(\mathcal{O}_{K}^{\times} \hookrightarrow (\mathcal{O}_{K}^{\times})^{\Sigma}) = \text{coker}(\mathcal{O}_{K,S}^{\times} \hookrightarrow \mathcal{H}^{\times}(K,S)).
$$
One verifies immediately from Dirichlet's unit theorem, that $\text{coker}(\mathcal{O}_K^{\times} \hookrightarrow (\mathcal{O}_K^{\times})^{\Sigma})$ is $n$-torsion-free. On the other hand, the image of $a$ in $\text{coker}(\mathcal{O}_{K,S}^{\times} \hookrightarrow \mathcal{H}^{\times}(K,S))$ is an $n$-torsion element, hence trivial, which implies that $a \in \mathcal{O}_{K,S}^{\times}$.
\end{proof}

\begin{defn}
Let $K$ be a number field, and let $S$ be a set of primes of $K$ such that $\mathscr{P}_K^{\text{inf}},\mathscr{P}_K^{\bullet} \subset S$ and $\delta(S) = 1$ and $S$ is conjugate-stable. Let $L/K$ be a finite extension contained in $K_S$, we write $S_L$ for the set of primes of $L$ that lying above primes in $S$. Moreover, we write $\Sigma$ for the set of prime numbers determined by $S$. For each $i \geq 1$, we construct the following two sets
$$
\mathcal{G}(L,i) \subset \mathcal{F}(L,i) \subset \prod_{\mathfrak{p} \in S_L^{d=1}} L_{\mathfrak{p}}
$$
as follows:
\begin{itemize}
\item For $i := 1$, we define $\mathcal{G}(L,i) := \mathcal{F}(L,i) := \{1\}$.
\item For $i \geq 2$, we define
$$
\mathcal{G}(L,i) := \{a \in \mathcal{H}_{\times}(L,S_L): \exists n \in \mathbb N(\Sigma)~\text{s.t.}~a^n \in \mathcal{F}(L,i-1)\}.
$$
\item For $i \geq 2$, we define $\mathcal{F}(L,i)$ to be the subring of $\prod_{\mathfrak{p} \in S_L^{d=1}} L_{\mathfrak{p}}$ generated by $\mathcal{G}(L,i)$.
\end{itemize}
Moreover, we write $\mathcal{F}_K(i) := \varinjlim_L \mathcal{F}(L,i)$ where $L$ ranges over all finite extensions $L/K$ contained in $K_S$. Moreover, we write $\mathcal{F}_K := \bigcup_{i \geq 0} \mathcal{F}_K(i)$.
\end{defn}
\begin{lem}
The equality
$$
\mathcal{F}_K = \mathbb Q_{\Sigma}^{\text{sol}}
$$
holds true in $K_S$ (the left hand side of the equality is contained in $K_S$ by Lemma 4.2 and the right hand side of the equality is contained in $K_S$ by the conjugate-stability condition).
\end{lem}
\begin{proof}
Consider $i = 2$, by Lemma 3.11 (ii) and Lemma 4.2, we have
$$
\mathcal{F}_K(2) = \mathbb Q(\mu(K_S)) = \bigcup_{m \in \mathbb N(\Sigma)} \mathbb Q(\zeta_m)
$$
which is abelian and unramified outside $\Sigma$ over $\mathbb Q$, hence contained in $\mathbb Q_{\Sigma}^{\text{sol}}$. Assume that for $i \geq 3$, the containment $\mathcal{F}_K(i) \subset \mathbb Q_{\Sigma}^{\text{sol}}$ holds true. Let us consider $\mathcal{F}_K(i+1)$. By Definition 4.3 and the equality $\mathcal{F}_K(2) = \mathbb Q(\mu(K_S)) \subset \mathcal{F}_K(i)$, $\mathcal{F}_K(i+1)$ is an abelian extension of $\mathcal{F}_K(i)$ which is also unramified outside $\Sigma$ over $\mathbb Q$, hence also the containment $\mathcal{F}_K(i+1) \subset \mathbb Q_{\Sigma}^{\text{sol}}$ holds true. Thus the containment $\mathcal{F}_K \subset \mathbb Q_{\Sigma}^{\text{sol}}$ holds true by induction. \par 
Now we verify the converse. Let $F/\mathbb Q$ be a finite abelian unramified outside $\Sigma$ extension. Without the loss of generality, assume that $F$ contains suitable $n$-th roots of unity for some $n \in \mathbb N(\Sigma)$ and that $M = F(a_1^{1/r_1},\dots,a_s^{1/r_s})$ is abelian and unramified outisde $\Sigma$ where $a_1,\dots,a_s \in F^{\times}$ and $r_1,\dots,r_s \in \mathbb N(\Sigma)$ and $s \in \mathbb N$. We write $S_F$ for the set of all primes of $F$ lying above, i.e. $S_F = \{\mathfrak{p}\in \mathscr{P}_F: \text{char}(\mathfrak{p})|\ell~\text{for some}~\ell \in \Sigma\} \cup \mathscr{P}_F^{\text{inf}}$. It follows from Kummer theory, that we have the following exact sequence
$$
1 \to \mathcal{O}_{F,S_F}^{\times}/(\mathcal{O}_{F,S_F}^{\times})^n \to H^1(G_{F,S_F},\mu_n) \to \text{Cl}_{S_F}(F)[n] \to 1
$$
where $\text{Cl}_{S_F}(F) = 1$ because $\delta(S_F) = 1$. In other words, in order to construct $M$, it suffices to take $a_1,\dots,a_s \in \mathcal{O}_{F,S_F}^{\times}$. By passing through the limit, we can conclude that every finite abelian unramified outside $\Sigma$ extension of $\bigcup_{m \in \mathbb N(\Sigma)} \mathbb Q(\zeta_m)$ is in the form
$$
(\bigcup_{m \in \mathbb N(\Sigma)} \mathbb Q(\zeta_m))( a_1^{1/t_1},\dots,a_u^{1/t_u})
$$
where $u \in \mathbb N$, $t_1,\dots,t_u \in \mathbb N(\Sigma)$ and 
$$
a_1,\dots,a_u \in \varinjlim_{m \in \mathbb N(\Sigma)}~ \mathcal{O}_{\mathbb Q(\zeta_m),S_{\mathbb Q(\zeta_m)}}^{\times}  \subset \varinjlim_L~\mathcal{O}^{\times}_{L,S_L}
$$
for $L$ ranges over all finite extensions of $K$ contained in $K_S$. In particular, every finite abelian unramified outside $\Sigma$ extension of $\bigcup_{m \in \mathbb N(\Sigma)}\mathbb Q(\zeta_m)$ is contained in $\mathcal{F}_K(2)$. Then it follows from induction on $i$, in the similar way as in the first half of the proof, we may conclude that $\mathcal{F}_k \supset \mathbb Q_{\Sigma}^{\text{sol}}$. This proves Lemma 4.4.
\end{proof}

\begin{prop}
Let $G$ be a profinite group of NF-type with density 1 restricted ramification. Then there is a group-theoretic reconstruction of the following objects from $G$: \par 
(i) For each $v \in S^{d=1}(G)$, the topological rings
$$
\mathcal{O}(v) ~\text{and}~k(v)
$$
which are isomorphic to $\mathbb Z_{p_v}$ and $\mathbb Q_{p_v}$ respectively. \par 
(ii) The field
$$
Q(G)
$$
which is isomorphic to $\mathbb Q$.
\end{prop}
\begin{proof}
We apply the same constructions in Definition 2.1, Remark 2.1.1 and Definition 2.2 in \cite{Ho2}, to establish a group-theoretic reconstruction of objects in this proposition. 
\end{proof}
\begin{defn}
Let $G$ be a profinite group of NF-type with density 1 restricted ramification. We define the following two sets
$$
\mathcal{G}(H,i) \subset \mathcal{R}(H,i) \subset \prod_{v \in S^{d=1}(H)} k(v)
$$
associated to each open subgroup $H \subset G$ as follows:
\begin{itemize}
    \item For $i = 1$, we define
    $$
    \mathcal{G}(H,i) := \mathcal{F}(H,i) : =\{1\} \subset \mathcal{H}_{\times}(H,S(H)).
    $$
    \item For $i \geq 2$, we define
    $$
    \mathcal{G}(H,i) := \{a \in \mathcal{H}_{\times}(H,S(H)): \exists n \in \mathbb{N}(\Sigma(G))~\text{s.t.}~a^n \in \mathcal{F}(H,i-1)\}. 
    $$
    \item For $ i \geq 2$, we define
    $$
    \mathcal{F}(H,i) \subset \prod_{v \in S^{d=1}(H)} k_{\times}(v)
    $$
    for the subring generated by $\mathcal{G}(H,i)$ (the inclusion holds true by Lemma 3.11 (iii)).
\end{itemize}
Moreover, we write $\mathcal{F}(G,i) := \varinjlim_H \mathcal{F}(H,i)$ for $i \geq 1$ where the transition maps are induced by the left vertical arrow in Proposition 3.14. Moreover, we write $\mathcal{F}(G) := \bigcup_{i \geq 1} \mathcal{F}(G,i)$.
\end{defn}
\begin{prop}
Let $G$ be a profinite group of NF-type with density 1 restricted ramification. Then there exists an isomorphism of fields
$$
\mathcal{F}(G) \xrightarrow{\sim} \mathbb Q^{\text{sol}}_{\Sigma}
$$
which is equivariant w.r.t $\alpha: G \xrightarrow{\sim} G_{K,S}$.
\end{prop}
\begin{proof}
This is follows from the construction of $\mathcal{F}(G)$ and Proposition 3.14 and Lemma 4.4.
\end{proof}
\begin{cor}
Let $G$ be a profinite group of NF-type with density 1 restricted ramification. Then there exists an isomorphism of profinite groups
$$
\text{Aut}_{\text{fields}}(\mathcal{F}(G)) \xrightarrow{\sim} G_{\mathbb Q,\Sigma}^{\text{sol}}.
$$
\end{cor}
\begin{proof}
This is an immediate consequence of Proposition 4.7.
\end{proof}
We shall write
$$
\Gamma(G) := \text{Aut}_{\text{fields}}(\mathcal{F}(G)).
$$
For the rest of this section, we shall develop a group-theoretic reconstruction of the set of decomposition subgroups of $G_{\mathbb Q,\Sigma}^{\text{sol}}$ at non-archimedean primes of $\mathbb Q^{\text{sol}}_{\Sigma}$ lying above non-archimedean primes in $\Sigma$ where $\delta(\Sigma) = 1$ and contains $2$ and all archimedean primes.
\begin{prop}
Let $K$ be a number field, and $S \subset \mathscr{P}_K$ be such that $\delta(S) = 1$ and $\mathscr{P}_K^{\text{inf}},\mathscr{P}_K^{\bullet} \subset S$, moreover, we write $\Sigma$ for the set of prime numbers determined by $S$. Then for $i = 1,2$, the following inflation maps
$$
\text{inf}:H^i(G_{K,S}^{\text{sol}},\mathbb Z/n\mathbb Z(j)) \to H^i(G_{K,S},\mathbb Z/n\mathbb Z(j))
$$
are isomorphisms, where $j = 0 ~\text{or} ~1$ and $n \in \mathbb N(S)$.
\end{prop}
\begin{proof}
For $i = 1$, consider the following commutative diagram with exact columns arising from inflation-restriction sequence:
$$
\begin{tikzcd}
	1 && 1 \\
	{H^1(G_{K,S}^{\text{ab}},\mathbb Z/n\mathbb Z(j)^{\Delta})} && {H^1(G_{K,S}^{\text{ab}},\mathbb Z/n\mathbb Z(j)^{\tilde \Delta} )} \\
	{H^1(G_{K,S}^{\text{sol}},\mathbb Z/n\mathbb Z(j))} && {H^1(G_{K,S},\mathbb Z/n\mathbb Z(j))} \\
	{H^1(\Delta,\mathbb Z/n\mathbb Z(j))^{G_{K,S}^{\text{ab}}}} && {H^1(\tilde \Delta,\mathbb Z/n \mathbb Z(j))^{G_{K,S}^{\text{ab}}}} \\
	{H^2(G_{K,S}^{\text{ab}},\mathbb Z/n\mathbb Z(j)^{\Delta})} && {H^2(G_{K,S}^{\text{ab}},\mathbb Z/n\mathbb Z(j)^{\tilde \Delta})}
	\arrow[from=1-1, to=2-1]
	\arrow[from=1-3, to=2-3]
	\arrow["\sim", from=2-1, to=2-3]
	\arrow[from=2-1, to=3-1]
	\arrow[from=2-3, to=3-3]
	\arrow["{\text{inf}}", from=3-1, to=3-3]
	\arrow[from=3-1, to=4-1]
	\arrow[from=3-3, to=4-3]
	\arrow["\sim", from=4-1, to=4-3]
	\arrow[from=4-1, to=5-1]
	\arrow[from=4-3, to=5-3]
	\arrow["\sim", from=5-1, to=5-3]
\end{tikzcd}
$$
where $\Delta := \text{Gal}(K_{S}^{\text{sol}}/K_{S}^{\text{ab}})$ and $\tilde \Delta := \text{Gal}(K_{S}/K_{S}^{\text{ab}})$. In the commutative diagram above, the first horizontal arrow is an isomorphism because all $n$-th roots of unity are contained in $K_{S}^{\text{ab}}$ for $n \in \mathbb N(S)$, hence the action of $\Delta$ and $\tilde \Delta$ on $\mathbb Z/n\mathbb Z(j)$ is trivial for $j = 0,1$ and hence the fourth horizontal arrow is also an isomorphism. The third horizontal arrow is an isomorphism because both $\Delta$ and $\tilde \Delta$ have the same abelianisation together with the fact that $\Delta,\tilde \Delta$ act trivially on $\mathbb Z/n\mathbb Z(j)$. Hence it follows from easy diagram chasing, that
$$
\text{inf}: H^1(G_{K,S}^{\text{sol}},\mathbb Z/n\mathbb Z(j)) \to H^1(G_{K,S},\mathbb Z/n\mathbb Z(j))
$$
is an isomorphism. \par 
For $i = 2$, consider the following commutative diagram with exact columns induced by the Hochschild-Serre spectral sequence (e.g. see section 1 in \cite{S-extn}):
$$
\begin{tikzcd}
	{H^1(\Delta,\mathbb Z/n\mathbb Z(j))^{G_{K,S}^{\text{ab}}}} && {H^1(\tilde \Delta,\mathbb Z/n\mathbb Z(j))^{G_{K,S}^{\text{ab}}}} \\
	{H^2(G_{K,S}^{\text{ab}},\mathbb Z/n\mathbb Z(j)^{\Delta})} && {H^2(G_{K,S}^{\text{ab}},\mathbb Z/n\mathbb Z(j)^{\tilde \Delta})} \\
	{H^2(G_{K,S}^{\text{sol}},\mathbb Z/n\mathbb Z(j))_1} && {H^2(G_{K,S},\mathbb Z/n\mathbb Z(j))_1} \\
	{H^1(G_{K,S}^{\text{ab}},H^1(\Delta,\mathbb Z/n\mathbb Z(j) ))} && {H^1(G_{K,S}^{\text{ab}},H^1(\tilde \Delta,\mathbb Z/n\mathbb Z(j)))} \\
	{H^3(G_{K,S}^{\text{ab}},\mathbb Z/n\mathbb Z(j)^{\Delta})} && {H^3(G_{K,S}^{\text{ab}},\mathbb Z/n\mathbb Z(j)^{\tilde \Delta})}
	\arrow["\sim", from=1-1, to=1-3]
	\arrow[from=1-1, to=2-1]
	\arrow[from=1-3, to=2-3]
	\arrow["\sim", from=2-1, to=2-3]
	\arrow[from=2-1, to=3-1]
	\arrow[from=2-3, to=3-3]
	\arrow["{\text{inf}}", from=3-1, to=3-3]
	\arrow[from=3-1, to=4-1]
	\arrow[from=3-3, to=4-3]
	\arrow["\sim", from=4-1, to=4-3]
	\arrow[from=4-1, to=5-1]
	\arrow[from=4-3, to=5-3]
	\arrow["\sim", from=5-1, to=5-3]
\end{tikzcd}
$$
Where $H^2(G_{K,S}^{\text{sol}},\mathbb Z/n \mathbb Z(j))_1$ (resp. $H^2(G_{K,S},\mathbb Z/n\mathbb Z(j))$) is defined to be 
$$
\text{ker}(H^2(G_{K,S}^{\text{sol}},\mathbb Z/n\mathbb Z(j)) \xrightarrow{\text{res}} H^2(\Delta,\mathbb Z/n\mathbb Z(j))^{G_{K,S}^{\text{ab}}})
$$ 
(resp. 
$$
\text{ker}(H^2(G_{K,S},\mathbb Z/n\mathbb Z(j)) \xrightarrow{\text{res}} H^2(\tilde \Delta,\mathbb Z/n\mathbb Z(j))^{G_{K,S}^{\text{ab}}})
$$) and all horizontal arrows are isomorphisms (the first and the second horizontal arrows are isomorphisms in the first commutative diagram, the fourth and the fifth horizontal arrows are isomorphisms because both are identity maps, hence we may conclude the middle horizontal arrow is also an isomorphism). \par 
Moreover, we have the following commutative diagram with exact rows by using Hochschild-Serre spectral sequence:
$$
\begin{tikzcd}
	1 & {H^2(G_{K,S}^{\text{sol}},\mathbb Z/n\mathbb Z(j))_1} & {H^2(G_{K,S}^{\text{sol}},\mathbb Z/n\mathbb Z(j))} & {H^2(\Delta,\mathbb Z/n\mathbb Z(j))^{G_{K,S}^{\text{ab}}}} \\
	1 & {H^2(G_{K,S},\mathbb Z/n\mathbb Z(j))_1} & {H^2(G_{K,S},\mathbb Z/n\mathbb Z(j))} & {H^2(\tilde \Delta,\mathbb Z/n\mathbb Z(j))^{G_{K,S}^{\text{ab}}}}
	\arrow[from=1-1, to=1-2]
	\arrow[from=1-2, to=1-3]
	\arrow[from=1-2, to=2-2]
	\arrow[from=1-3, to=1-4]
	\arrow[from=1-3, to=2-3]
	\arrow[from=1-4, to=2-4]
	\arrow[from=2-1, to=2-2]
	\arrow[from=2-2, to=2-3]
	\arrow[from=2-3, to=2-4]
\end{tikzcd}
$$
where the vertical arrows are inflation maps, and the left vertical arrow is an isomorphism and the middle vertical arrow is injective. Now we show that $H^2(\tilde \Delta,\mathbb Z/n\mathbb Z(j))$ is trivial. \par 
Without the loss of generalities, we may assume that $j = 1$. We write $\Gamma := \text{Gal}(\overline{K}/K_{S}^{\text{ab}})$, $J := \text{ker}(\Gamma \twoheadrightarrow \Delta)$ and $\tilde J := \text{ker}(\tilde \Delta \twoheadrightarrow \Delta)$. Then the Hochschild-Serre spectral sequence induces the following commutative diagram with exact rows: 
$$
\begin{tikzcd}
	{H^1(\tilde J,\mu_n)^{\Delta}} & {H^2(\Delta,\mu_n^{\tilde J})} & {H^2(\tilde \Delta,\mu_n)} \\
	{H^1(J,\mu_n)^{\Delta}} & {H^2(\Delta,\mu_n^{J})} & {H^2(\Gamma,\mu_n)}
	\arrow[from=1-1, to=1-2]
	\arrow[from=1-1, to=2-1]
	\arrow[from=1-2, to=1-3]
	\arrow[from=1-2, to=2-2]
	\arrow[from=1-3, to=2-3]
	\arrow[from=2-1, to=2-2]
	\arrow[from=2-2, to=2-3]
\end{tikzcd}
$$
where the left vertical arrow is an isomorphism since both $\tilde J$ and $J$ have isomorphic maximal abelian exponent dividing $n$ quotients and $n \in \mathbb N(S)$, the middle vertical arrow is an isomorphism because it is the identity map. Thus, we may conclude that the right vertical arrow is injective. It follows from Chapter II 3.3 Proposition 9 in \cite{SerreGal} that $\text{cd}_p(\Gamma) \leq 1$ for each $p \in \Sigma$, which implies that $H^2(\Gamma,\mu_n) = 1$. Hence we may also conclude that $H^2(\tilde \Delta,\mu_n) = 1$. \par 
Then consider the following commutative diagram arising from the Hochschild-Serre spectral sequence
$$
\begin{tikzcd}
	{H^2(G_{K,S}^{\text{sol}},\mathbb Z/n\mathbb Z(j))_1} && {H^2(G_{K,S}^{\text{sol}},\mathbb Z/n\mathbb Z(j))} \\
	\\
	{H^2(G_{K,S},\mathbb Z/n\mathbb Z(j))_1} && {H^2(G_{K,S},\mathbb Z/n\mathbb Z(j))}
	\arrow[hook, from=1-1, to=1-3]
	\arrow["\wr"', from=1-1, to=3-1]
	\arrow[hook, from=1-3, to=3-3]
	\arrow["\sim", from=3-1, to=3-3]
\end{tikzcd}
$$
hence one can conclude that the right vertical map is an isomorphism. This completes the proof of Proposition 4.9.
\end{proof}
\begin{prop}
We write $\widetilde{S}$ for the set of primes in $K_{S}^{\text{sol}}$ lying above $S$. Then there is a group-theoretic reconstruction to the $G_{K,S}^{\text{sol}}$-set:
$$
\text{Dec}(K_{S}^{\text{sol}}/K) := \{D_{\tilde{\mathfrak{p}}} \subset G_{K,S}^{\text{sol}}: \tilde{\mathfrak{p}} \in \widetilde{S}\}
$$
starting from $G_{K,S}^{\text{sol}}$.
\end{prop}
\begin{proof}
By Proposition 4.9, one checks easily that Theorem 9.4.3 in \cite{NSW} still holds true. In particular, elements in $\text{Dec}(K_S^{\text{sol}}/K)$ are full. In paritulcar, the composite
$$
D \hookrightarrow G_{K,S} \twoheadrightarrow G_{K,S}^{\text{sol}}
$$
is injective for any decomposition group $D \subset G_{K,S}$ at primes in $S_f(K_S)$. Then by Corollary 2.7 (ii) in \cite{Ivanov}, that we can conclude that $\text{Dec}(K_S^{\text{sol}}/K)$ can be characterised as the set of maximal closed subgroups of $G_{K,S}^{\text{sol}}$ of MLF-type (c.f. Definition 1.1 in \cite{Ho1}).
\end{proof}
Hence, from $\Gamma(G)$, we may group-theoretic recover the set $\widetilde{\text{Dec}}(\Gamma(G)) := \{\text{maximal closed subgroups of $\Gamma(G)$ of MLF-type}\}$(c.f. Definition 1.1 in \cite{Ho1}) and $\text{Dec}(\Gamma(G)) := \widetilde{\text{Dec}}(\Gamma(G))^{\Gamma(G)} \xrightarrow{\sim} \Sigma(G)$ fits into the following commutative diagram
$$
\begin{tikzcd}
	{\widetilde{\text{Dec}}(\Gamma(G))} & {\widetilde{\Sigma}} \\
	{\text{Dec}(\Gamma(G))} & \Sigma
	\arrow["\sim", from=1-1, to=1-2]
	\arrow[two heads, from=1-1, to=2-1]
	\arrow[two heads, from=1-2, to=2-2]
	\arrow["\sim", from=2-1, to=2-2]
\end{tikzcd}
$$
where the top horizontal arrow is a bijection equivariant w.r.t $\Gamma(G) \xrightarrow{\sim} G_{\mathbb Q,\Sigma}^{\text{sol}}$, horizontal surjections are determined by modding the action by conjugation, and the bottom arrow is a bijection.
\begin{prop}
Let $G$ be a profinite group of NF-type with density 1 restricted ramification. Then the natural action of $G$ on $\mathcal{F}(G)$ (c.f. Definition 4.6) determines a natural map
$$
\phi_G:G \to \Gamma(G)
$$
such that for each $D \in \widetilde{S}(G)$, $\phi_G(D)$ coincide with some open subgroup $H \subset \mathfrak{D}$, where $\mathfrak{D} \in \widetilde{\text{Dec}}(\Gamma(G))$.
\end{prop}
\begin{proof}
Since $G$ acts on $\mathcal{F}(G)$ by automorphism, then one defines a well-defined map
$$
\phi_G: G \to \Gamma(G).
$$
Moreover, $\phi_G$ factors through $G^{\text{sol}}$. Let $D \in \widetilde{S}(G)$, it follows from Proposition 4.10, that the composite
$$
D \hookrightarrow G \twoheadrightarrow G^{\text{sol}} \to \Gamma(G)
$$
is injective. Finally, since $D$ is of MLF-type, hence $\phi_G(D)$ can be viewed as an open subgroup of $\mathfrak{D} \in \widetilde{\text{Dec}}(\Gamma(G))$, for $p(\mathfrak{D}) = p(D)$.
\end{proof}
\section{Reconstruction of Number Fields}
In this section, we establish a group-theoretic reconstruction to the field $K_S$ starting from $G(\xrightarrow{\sim} G_{K,S})$ by applying the strategy in \cite{Ho1} and \cite{Ho2}. In particular, Theorem 2.4 in \cite{Shi2} plays an essential role in the arguments.
\begin{prop}
Let $G$ be a profinite group of NF-type with density 1 restricted ramification, and let $D \in \widetilde{S}(G)$. Then there is a group-theoretic reconstruction of a field structure on the multiplicative monoid $\bar k_{\times}(D)$ such that the following diagram commutes
$$
\begin{tikzcd}
	{\bar k(D)} & {\overline{\mathbb Q_{p(D)}}} \\
	{k(v)} & {K_{\mathfrak p}}
	\arrow["\sim", from=1-1, to=1-2]
	\arrow[hook', from=2-1, to=1-1]
	\arrow["\sim", from=2-1, to=2-2]
	\arrow[hook', from=2-2, to=1-2]
\end{tikzcd}
$$
where $\bar k(D)$ is defined to be $\bar k_{\times}(D)$ equipped with a field structure, $k(v) := \bar k(D)^D$, and vertical arrows are field embeddings, and the horizontal arrows are isomorphisms, and $\mathfrak{p} \in S$ corresponds to the conjugacy class $v \in S(G)$ of $D$. 
\end{prop}
\begin{proof}
This follows immediately from Proposition 4.7 and Proposition 5.8 in \cite{Ho1}, together with Theorem 9.4.3 in \cite{NSW}.
\end{proof}
Next, we construct a similar object as in Definition 3.3 in \cite{Ho2}.
\begin{defn}
Let $G$ be a profinite group of NF-type with density 1 restricted ramification. For each $D \in \widetilde{S}(G)$, we call a pair of subfields $F[D] \subset F_S[D] \subset \bar k(D)$ an $S(G)$-standard pair if the followings hold: \par
(i) $F[D]$ is a finite extension of the prime subfield of $\bar k(D)$. \par 
(ii) $F_S[D]$ is the maximal unramified outside $S[D]$-extension of $F[D]$ contained in $\bar k(D)$ for some density 1 set of primes $S[D]$ of $F[D]$. \par 
(iii) It holds that every automorphism $\sigma : \bar k(D) \xrightarrow{\sim} \bar k(D)$ induces an automorphism $\sigma: F_S[D] \xrightarrow{\sim} F_S[D]$ that restricts to the identity automorphism on $F[D]$. \par 
(iv) There exists an isomorphism $\varphi_D: \text{Gal}(F_S[D]/F[D]) \xrightarrow{\sim} G$ such that the following diagram commutes:
$$
\begin{tikzcd}
	{\text{Gal}(F_S[D]/F[D])} && G \\
	& D.
	\arrow["{\varphi_D}"', from=1-1, to=1-3]
	\arrow[hook', from=2-2, to=1-1]
	\arrow[hook, from=2-2, to=1-3]
\end{tikzcd}
$$
\end{defn}
\begin{rem}
(a) It follows from Lemma 3.5 in \cite{Ho2} and Theorem 2.4 in \cite{Shi2} that, the isomorphism $\varphi_D$ in Definition 5.2 (iv) is unique. \par 
(b) Notice that when working with density 1 set of primes, the conditions presented in Theorem 2.4 in \cite{Shi2} are automatic. \par 
\end{rem}
\begin{prop}
Let $G$ be a profinite group NF-type with density 1 restricted ramification. Let $D \in \widetilde{S}(G)$. Then $S(G)$-standard pair subfields of $\bar k(D)$ exists and is unique. In particular, there is a group-theoretic reconstruction of the $S(G)$-standard pair subfields of $\bar k(D)$.
\end{prop}
\begin{proof}
The existence of such standard pair subfields of $\bar k(D)$ is immediate. We check the uniqueness. \par 
Let $F_S[D]_1/F[D]_1$ and $F_S[D]_2/F[D]_2$ be $S(G)$-standard pair subfields of $\bar k(D)$. By Definition 5.2 and Remark 5.3 (a), there exists a unique isomorphism
$$
\varphi_2^{-1} \circ \varphi_1: \text{Gal}(F_S[D]_1/F[D]_1) \xrightarrow{\sim} \text{Gal}(F_S[D]_2/F[D]_2)
$$
where $\varphi_i := \text{Gal}(F_S[D]_i/F[D]_i) \xrightarrow{\sim} G$ is the unique isomorphism in Definition 5.2 (iv). Then it follows from Theorem 2.4 in \cite{Shi2}, that $\varphi_2^{-1}\circ \varphi_1$ corresponds to a unique field $D$-equivariant isomorphism
$$
\tau: F_S[D]_2 \xrightarrow{\sim} F_S[D]_1
$$
that restricts to an isomorphism $\tau_0: F[D]_2 \xrightarrow{\sim} F[D]_1$. \par
Now we shall denote by $F^{\text{al}}[D]_i$ for the algebraic closure of $F[D]_i$ in $\bar k(D)$ for $i = 1,2$. Then it holds that we have the following commutative diagram (i.e. every triangle in the following diagram is commutative)
$$
\begin{tikzcd}
	& {\text{Gal}(F^{\text{al}}[D]_i/F[D]_i)} \\
	& D \\
	{\text{Gal}(F_S[D]_i/F[D]_i)} && G
	\arrow[two heads, from=1-2, to=3-1]
	\arrow[two heads, from=1-2, to=3-3]
	\arrow[hook, from=2-2, to=1-2]
	\arrow[hook, from=2-2, to=3-1]
	\arrow[hook, from=2-2, to=3-3]
	\arrow["{\varphi_i}", from=3-1, to=3-3]
\end{tikzcd}
$$
where $D \hookrightarrow \text{Gal}(F^{\text{al}}[D]_i/F[D]_i)$ is the natural inclusion, the left surjection $\text{Gal}(F^{\text{al}}[D]_i/F[D]_i)$ is obtained by restricting an automorphism of $F^{\text{al}}[D]_i$ to $F_S[D]_i$, and the right surjection is the unique surjection makes the diagram commutes. Then there exists a unique isomorphism
$$
\rho: \text{Gal}(F^{\text{al}}[D]_1/F[D]_1) \xrightarrow{\sim} \text{Gal}(F^{\text{al}}[D]_2/F[D]_2).
$$
It follows from the Neukirch-Uchida theorem, that $\rho$ induces a unique $D$-equivariant isomorphism
$$
\theta: F^{\text{al}}[D]_2 \xrightarrow{\sim} F^{\text{al}}[D]_1
$$
which restricts to an isomorphism $F[D]_2 \xrightarrow{\sim} F[D]_1$. On the other hand, Lemma 3.6 (ii) in \cite{Ho2} implies that the isomorphism $\theta$ is the identity map hence also $F[D]_1 \xrightarrow{\sim} F[D]_2$ is the identity map. Furthermore, by considering the following commutative diagram
$$
\begin{tikzcd}
	{\text{Gal}(F^{\text{al}}[D]_1/F[D]_1)} && {\text{Gal}(F^{\text{al}}[D]_2/F[D]_2)} \\
	{\text{Gal}(F_S[D]_1/F[D]_1)} && {\text{Gal}(F_S[D]_2/F[D]_2)}
	\arrow["\rho", from=1-1, to=1-3]
	\arrow[two heads, from=1-1, to=2-1]
	\arrow[two heads, from=1-3, to=2-3]
	\arrow["{\varphi_2^{-1}\circ\varphi_1}"', from=2-1, to=2-3]
\end{tikzcd}
$$
which induces the following commutative diagram of fields
$$    
\begin{tikzcd}
	{F^{\text{al}}[D]_2} && {F^{\text{al}}[D]_1} \\
	{F_S[D]_2} && {F_S[D]_1}
	\arrow["\theta", from=1-1, to=1-3]
	\arrow[hook, from=2-1, to=1-1]
	\arrow["\tau", from=2-1, to=2-3]
	\arrow[hook, from=2-3, to=1-3]
\end{tikzcd}
$$
hence we can conclude that the isomorphism $\tau: F_S[D]_2 \xrightarrow{\sim} F_S[D]_1$ is the identity map. This completes the proof of Proposition 5.4.
\end{proof}
\begin{rem}
In the proof of Proposition 5.4, we have essentially recovered group-theoretically the profinite group $\text{Gal}(F^{\text{al}}[D]/F[D]) \twoheadrightarrow G$, which may be thought of as an analogue of cuspidalisation theory developed in \cite{Mzk2} and \cite{Mzk4}. \par 
On the other hand, take distinct $D,E \in \widetilde{S}(G)$, although the profinite groups $\text{Gal}(F^{\text{al}}[D]/F[D])$ and $\text{Gal}(F^{\text{al}}[E]/F[E])$ are abstractly isomorphic. At this moment, the authors do not know how to construct a canonical isomorphism $\text{Gal}(F^{\text{al}}[D]/F[D]) \xrightarrow{\sim} \text{Gal}(F^{\text{al}}[E]/F[E])$ starting from the profinite group $G (\xrightarrow{\sim} G_{K,S})$. In particular, to construct such a canonical isomorphism, one has to characterise the inertia subgroups of $G$ outside $\widetilde{S}(G)$ group-theoretically (i.e. the "cuspidal inertia groups"). 
\end{rem}
\begin{thm}
Let $G$ be a profinite group of NF-type with density 1 restricted ramification. Then there is a group-theoretic reconstruction of the field $F_S(G)$ and the subfield $F(G) := F_S(G)^G$ such that the following diagram commutes
$$
\begin{tikzcd}
F_S(G) \arrow[r,"\sim"] & K_S \\
F(G) \arrow[u,hookrightarrow] \arrow[r,"\sim"] & K \arrow[u,hookrightarrow]
\end{tikzcd}
$$
where the upper horizontal arrow is an isomorphism equivariant w.r.t $\alpha: G \xrightarrow{\sim} G_{K,S}$, vertical arrows are field embeddings and the lower horizontal arrow is an isomorphism. Moreover, the reconstruction $G \mapsto F_S(G)$ is functorial with respect to isomorphisms and is compatible with taking open subgroups.
\end{thm}
\begin{proof}
We shall use a similar construction as in \cite{Ho2}. Let $D,E \in \widetilde{S}(G)$, we write
$$
f_{D,E}: \text{Gal}(F_S[D]/F[D]) \xrightarrow{\sim} \text{Gal}(F_S[E]/F[E])
$$
where $f_{D,E} := \varphi_D \circ \varphi_E^{-1}$ c.f. Definition 5.2. Then by Theorem 2.4 in \cite{Shi2}, $f_{D,E}$ induces a unique isomorphism
$$
\gamma_{E,D}: F_S[E] \xrightarrow{\sim} F_S[D]
$$
that restricts to an isomorphism $\tilde \gamma_{E,D}: F[E] \xrightarrow{\sim} F[D]$. Then consider the subring
$$
F_S(G) \subset \prod_{D \in \widetilde{S}(G)} F_S[D]
$$
defined as
$$
F_S(G) := \{(\dots,a_D,\dots): \gamma_{D,E}(a_D) = a_E~\forall D,E \in \widetilde{S}(G)\}.
$$
In particular, the component-wise action of $G$ on $\prod_{D \in \widetilde{S}(G)}F_S[D]$ determined by the component-wise Galois action of $\text{Gal}(F_S[D]/F[D])$ on $F_S[D]$ via $\varphi_D^{-1}$ for all $\widetilde{S}(G)$ induces an action on $F_S(G)$. One checks immediately that this action preserves $F_S(G)$, indeed, since $G$ acts on each components of $F_S(G)$ via $\varphi_D^{-1}$, which corresponds to an element of $\text{Gal}(F_S[D]/F[D])$. \par 
Next we check that $F_S(G)$ is a field. Notice that by the construction of $F_S(G)$, we one verifies easily that if an element is not a unit, it must be a zero divisor. But by the definition of $F_S(G)$, a zero divisor of $F_S(G)$ must be component-wise zero divisor as well, which means there can be no non-zero zero divisors, together with the fact that $\gamma_{D,E}(a_D^{-1}) = a_E^{-1}$, we can conclude that $F_S(G)$ is a field. Furthermore, since the action of $G$ on $F_S(G)$ induces automorphisms on $F_S(G)$, hence $F(G) := F_S(G)^G$ is also a field such that $F_S(G)/F(G)$ is Galois with Galois group isomorphic to $G$. Hence we obtain an isomorphism
$$
\iota(G): G_{K,S} \xrightarrow{\sim} \text{Gal}(F_S(G)/F(G)).
$$
Hence the existence and the commutativity of the diagram
$$
\begin{tikzcd}
    F_S(G) \arrow[r,"\sim"] & K_S \\
    F(G) \arrow[u,hookrightarrow] \arrow[r,"\sim"] & K \arrow[u,hookrightarrow]
\end{tikzcd}
$$
follows from Theorem 2.4 in \cite{Shi2}. \par 
The functoriality of the reconstruction $G \mapsto F_S(G)$ follows from the assumption that $S \xrightarrow{\sim} S(G)$ is conjugate-stable (c.f. Definition 2.1) and various constructions involved. Finally, if $H \subset G$ is an open subgroup, then it follows from the condition $\mathscr{P}_K^{\bullet} \subset S$, together with the fact that open subgroup of NF-type with density 1 restricted ramification is again a profinite group of NF-type with density 1 restricted ramification.
\end{proof}
\begin{rem}
The reconstruction $G \mapsto F_S(G)$ uses Theorem 2.4 in \cite{Shi2} in an essential way, hence does not provide an alternative proof to Theorem 2.4 in \cite{Shi2}.
\end{rem}
\begin{rem}
In this paper, we are working with $\delta(S) = 1$. While Theorem 2.4 is proven for $0 < \delta(S) < 1$, the following technical issues stops us from proving a positive density version of Theorem 5.6: 
\begin{itemize}
    \item If $0<\delta(S)<1$, then the local-global map
    $$
    H^1(G_K,\mu_n) \to \prod_{\mathfrak{p} \in S_f} H^1(G_{K_{\mathfrak{p}}},\mu_n)
    $$
    is not necessarily injective. From this point of view, it is not clear how to construct a version of $\mathcal{H}^{\times}(K,S)$ to make section 4 work.
    \item If $0 <\delta(S)<1$, it is not clear that the decomposition groups of $G_{K,S}$ at primes in $S_f(K_S)$ are full or not. If the decomposition groups are not full, then Theorem 2.4 does not work.
\end{itemize}
\end{rem}
\bibliography{ref.bib}

@article{CC,
  title={Corps de nombres peu ramifi{\'e}s et formes automorphes autoduales},
  author={Gaetan Chenevier and Laurent Clozel},
  journal={Journal of the American Mathematical Society},
  year={2007},
  volume={22},
  pages={467-519},
  url={https://api.semanticscholar.org/CorpusID:115165487}
}

@article{Ram2,
     author = {Demb\'el\'e, Lassina},
     title = {A non-solvable {Galois} extension of $ \mathbb{Q}$ ramified at 2 only},
     journal = {Comptes Rendus. Math\'ematique},
     pages = {111--116},
     publisher = {Elsevier},
     volume = {347},
     number = {3-4},
     year = {2009},
     doi = {10.1016/j.crma.2008.12.004},
     language = {en},
     url = {https://www.numdam.org/articles/10.1016/j.crma.2008.12.004/}
}

@article{Ho1,
  author = "Hoshi, Yuichiro",
  title = "Mono-anabelian Reconstruction of Number Fields (On the examination and further development of inter-universal {T}eichm{\"u}ller theory)",
  journal = "RIMS K\^oky\^uroku Bessatsu",
  volume = "B76",
  year = "2019",
  url = "http://hdl.handle.net/2433/244782",
}

@article{Ho2,
  author = "Hoshi, Yuichiro",
  title = "Mono-anabelian Reconstruction of Solvably Closed {Galois} Extension of Number Fields",
  journal = "Journal of Mathematical Sciences, The University of Tokyo",
  volume = "29",
  year = "2022, No.3",
  url = "https://www.ms.u-tokyo.ac.jp/journal/abstract/jms290301.html",
}

@inbook{Ho4, place={Cambridge}, series={London Mathematical Society Lecture Note Series}, title={Conditional results on the birational section conjecture over small number fields}, booktitle={Automorphic Forms and Galois Representations}, publisher={Cambridge University Press}, author={Hoshi, Yuichiro}, editor={Diamond, Fred and Kassaei, Payman L. and Kim, MinhyongEditors}, year={2014}, pages={187–230}, collection={London Mathematical Society Lecture Note Series}}

@article{Ivanov,
    author = {Alexander Ivanov},
    title = {On some anabelian properties of arithmetic curves},
    journal = {Manuscripta Mathematica },
    year = {2014},
    url = {https://doi.org/10.1007/s00229-014-0664-z}
}

@article{Mzk2,
author = {Mochizuki, Shinichi},
title = {{Absolute anabelian cuspidalizations of proper hyperbolic curves}},
volume = {47},
journal = {Journal of Mathematics of Kyoto University},
number = {3},
pages = {451 -- 539},
year = {2007},
URL = {https://doi.org/10.1215/kjm/1250281022},
}

@article{Mzk4,
  title={TOPICS IN ABSOLUTE ANABELIAN GEOMETRY II: DECOMPOSITION GROUPS AND ENDOMORPHISMS},
  author={Shinichi Mochizuki},
  journal={Journal of Mathematical Sciences-the University of Tokyo},
  year={2013},
  volume={20},
  url={https://www.ms.u-tokyo.ac.jp/journal/abstract/jms200201.html}
}

@book{ANT,
   author = {Neukirch, J\"urgen},
   year = {1999},
   title = {Algebraic Number Theory},
   publisher = {Springer},
   address = {Berlin},
}

@book{NSW,
   author = {Neukirch, J\"urgen and Schmidt, Alexander and Wingberg, Kay},
   year = {2008},
   title = {Cohomology of Number Fields},
   publisher = {Springer},
   address = {Berlin},
}

@article{S-extn,
title = {Cohomology of split group extensions},
journal = {Journal of Algebra},
volume = {29},
number = {2},
pages = {255-302},
year = {1974},
issn = {0021-8693},
doi = {https://doi.org/10.1016/0021-8693(74)90099-4},
url = {https://www.sciencedirect.com/science/article/pii/0021869374900994},
author = {Chih-Han Sah}
}

@article{Shi1,
url = {https://doi.org/10.1515/crelle-2021-0090},
title = {The {N}eukirch–{U}chida theorem with restricted ramification},
title = {},
author = {Ryoji Shimizu},
pages = {187--217},
volume = {2022},
number = {785},
journal = {Journal für die reine und angewandte Mathematik (Crelles Journal)},
doi = {doi:10.1515/crelle-2021-0090},
year = {2022},
lastchecked = {2025-09-18}
}

@article{Shi2,
author = {Shimizu, Ryoji},
title = {Isomorphisms of Galois groups of number fields with restricted ramification},
journal = {Mathematische Nachrichten},
volume = {296},
number = {7},
pages = {3026-3033},
doi = {https://doi.org/10.1002/mana.202100438},
year = {2023}
}

@book{SerreGal,
author="Serre, Jean-Pierre",
Title="Galois Cohomology",
year="1997",
publisher="Springer Berlin Heidelberg",
address="Berlin, Heidelberg",
isbn="978-3-642-59141-9",
doi="10.1007/978-3-642-59141-9_2",
url="https://doi.org/10.1007/978-3-642-59141-9_2"
}
\bibliographystyle{alphaurl}
\end{document}